\documentclass{amsart}
\usepackage{tikz}

\usepackage{amsmath, amsthm, xypic, amssymb, enumerate, mathrsfs, eufrak, verbatim, graphicx, bbm, stackrel,MnSymbol}
\usepackage{color}
\usepackage[toc,page]{appendix}
\usepackage[mathscr]{euscript}
\usepackage[all,cmtip,2cell]{xy}
\xyoption{rotate}
\UseTwocells

\usepackage[all]{xy}
\usepackage{hyperref}
\usepackage{tikz}
\usepackage{tikz-3dplot}

\DeclareMathOperator*{\hc}{hocolim}

\DeclareMathOperator{\hofib}{hofib}

\DeclareMathOperator{\Sp}{\mathscr{S}p}
\DeclareMathOperator{\Hom}{Hom}
\DeclareMathOperator{\Fun}{Fun}
\DeclareMathOperator{\Mod}{Mod}

\newcommand{\cC}{\mathcal{C}}
\newcommand{\sD}{\mathcal{D}}

\newcommand{\cF}{\mathcal{F}}
\newcommand{\cG}{\mathcal{G}}

\newcommand{\cO}{\mathcal{O}}
\newcommand{\cP}{\mathcal{P}}

\newcommand{\cT}{\mathcal{T}}
\newcommand{\cX}{\mathcal{X}}

\newcommand{\DD}{\mathbb{D}}
\newcommand{\II}{\mathbb{I}}
\newcommand{\NN}{\mathbb{N}}

\newcommand{\PP}{\mathbb{P}}

\newcommand{\al}{\alpha}

\newcommand{\Si}{\Sigma}

\newcommand{\nulll}{\emptyset}

\newcommand{\sm}{\wedge}

\def\mapdown#1.{\Big\downarrow\rlap{$\vcenter{\hbox{$\scriptstyle#1$}}$}}

\newtheorem{thm}{Theorem}[section]
\newtheorem{prop}[thm]{Proposition}

\newtheorem{lemma}[thm]{Lemma}

\theoremstyle{definition}
\newtheorem{defn}[thm]{Definition}

\theoremstyle{definition}
\newtheorem{ex}[thm]{Example}

\theoremstyle{definition}
\newtheorem{rem}[thm]{Remark}

\theoremstyle{definition}

\theoremstyle{definition}

\newtheoremstyle{TheoremForIntro} 
        {.6em}{.6em}              
        {\itshape}                      
        {}                              
        {\bfseries}                     
        {. }                             
        { }                             
        {\thmname{#1}\thmnote{ \bfseries #3}}
    \theoremstyle{TheoremForIntro}
    \newtheorem{TheoremIntro}[thm]{Theorem}

\input xy
\xyoption{all}

\newcommand{\op}{\mathscr{O}}

\newcommand{\Id}{\text{id}}
\newcommand{\Symseq}{\text{SymmSeq}}
\newcommand{\Symfun}{\text{SymmFun}}

\title[A monoidal model for multilinearization]{A monoidal model for multilinearization}
\author{Sarah Yeakel}
\address{Department of Mathematics \\
University of Maryland \\
College Park, MD 20742}
\email{syeakel@math.umd.edu}

\begin{document}

\begin{abstract}
	Using the category of finite sets and injections, we construct a new model for the multilinearization of multifunctors between spaces that appears in the derivatives of Goodwillie calculus. We show that this model yields a lax monoidal functor from the category of symmetric functor sequences to the category of symmetric sequences of spaces after evaluating at $S^0$.  
\end{abstract}

\maketitle
\thispagestyle{empty}


The chain rule of calculus gives a way of relating the derivative of a composition to the derivatives of the composite functions; put another way, taking derivatives preserves composition of functions in some sense. Higher order derivatives satisfy a similar chain rule, discovered by Fa\`{a} di Bruno. This relationship between derivatives is a foundational tool for computations in any setting where derivatives make sense.

In \cite{G3}, a notion of derivative was introduced by Goodwillie for homotopical functors from spaces to spaces or spectra. He constructed a Taylor tower of polynomial approximations and identified the layers of this tower as infinite loop spaces, classified by certain spectra, called the derivatives of the functor. He further identified the homotopy type of the derivatives, showing that the $n$th derivative is equivalent to the multilinearization of the $n$th cross-effect of the functor evaluated at $S^0$. In \cite{kleinrognes}, Klein and Rognes gave a chain rule for Goodwillie's first derivatives establishing a non-manifold computation of the first derivative of a mapping space functor. In \cite{Chingthesis}, Ching showed that the higher derivatives of the identity functor of spaces form an operad by showing that the dual spectra form a cooperad, building on work of Johnson and Arone-Mahowald \cite{Johnsonthesis,AMpartitions}. Working with these models and using duality, Arone and Ching showed that the derivatives of other functors are bimodules over the derivatives of the identity \cite{ACchain}. From these module structures, they proved a Fa\`{a} di Bruno chain rule for the higher derivatives of functors. 

One question in \cite{ACchain} is whether the derivatives functor is lax monoidal as a functor from endofunctors of spaces to symmetric sequences, a categorical way of capturing the preservation of composition. This would account for the operad and module structures of derivatives and would extend the results of \cite{ACchain} to other functors. This note seeks to push closer to a positive answer to that question. 

Goodwillie's identification of the derivatives as multilinearized cross-effects evaluated at $S^0$ can be broken into a composition. Let $\cT$ be the category of pointed spaces, let $\cT^n$ be the cartesian product of $n$ copies of $\cT$, and let $\Fun(\cT,\cT)$ be the category of endofunctors of $\cT$. The first functor sends an endofunctor of spaces to its sequence of cross-effects. The second functor is multilinearization at each level followed by evaluation at $S^0$ in each variable.
\[
\xymatrix{ \Fun(\cT, \cT) \ar[rr]^{cr_\ast} && \prod_{n \geq 1} \Fun(\cT^n, \cT) \ar[rr]^{\DD_1^{(\ast)}-(S^0)} && \Symseq(\cT)    }
\]
We call an object $\cF_\ast$ of $\prod_{n \geq 1} \Fun(\cT^n,\cT)$ a symmetric functor sequence; it is a sequence of functors $\cF_k:\cT^k \to \cT$ which are symmetric in all $k$ variables. We say a functor is pointed if it preserves the one point space. The category of symmetric functor sequences has a monoidal product induced by composition of functors. The category of symmetric sequences in spaces also has a monoidal product called the composition product. We prove the following result, which says that the second functor is lax monoidal with respect to these structures. 

\begin{TheoremIntro}[\ref{linisoperad}]
	Given symmetric functor sequences $\cF_\ast, \cG_\ast$ which are pointed in each variable, there are natural maps $S^0 \to \DD_1\Id(S^0)$ and
\[
\DD_1^{(\ast)}\cG_\ast(S^0) \circ \DD_1^{(\ast)}\cF_\ast(S^0) \to \DD_1^{(\ast)}(\cG \circ \cF)_\ast(S^0)
\]
	where $\II$ is the category of finite sets and injections,
	$\DD_1$ is defined by
	\[
\DD_1 F (X) = \hc_{U \in \II} \Omega^U F(\Sigma^U X),
	\]
	and $\DD_1^{(n)}$ denotes applying $\DD_1$ in each of the $n$ variables.
	Further, for symmetric functor sequences satisfying extra connectivity conditions, $\DD_1^{(n)}$ is equivalent to ordinary multilinearization.
\end{TheoremIntro}

This theorem then reduces the question of \cite{ACchain} to: Is there a model for $cr_\ast: \Fun(\cT, \cT)\to \prod_{n \geq 1} \Fun_\ast(\cT^n, \cT)$ which is monoidal with respect to composition? 
The cross-effects functor is monoidal in some settings, for example, for functors of abelian categories \cite{JMderiving}.

In the first section, we review Goodwillie's definition of linearization and relevant definitions of operads, including the notion of symmetric functor sequences. In the second section, we discuss using the category of finite sets and injections to replace the straight line category as an index for homotopy colimits. We also review the definition of the sphere operad of \cite{AKsphere} and a modification which produces useful maps in linearizations. In the third section, we apply the new linearization to symmetric multifunctors, recovering the main theorem and the corollary that under certain connectivity conditions, the multilinearization of a functor-operad is an operad.

The work in this paper forms a portion of the author's PhD thesis written at the University of Illinois under the supervision of Randy McCarthy, to whom considerable thanks are due for insight, encouragement, and patience. The author would also like to thank Greg Arone for helping correct multiple mistakes in the first draft.

\section{Background and Conventions}

In this section, we review the notion of linearization in functor calculus and relevant definitions.

\begin{defn}
	Let $F:\cT \to \cT$ be an endofunctor of pointed spaces. We say $F$ is a \emph{homotopy functor} if it preserves weak equivalences. We say $F$ is \emph{reduced} if $F(\ast) \simeq \ast$, and \emph{pointed} if $F(\ast)=\ast$. Finally, $F$ is \emph{continuous} if the natural map $\Hom(X,Y) \to \Hom(F(X), F(Y))$ is a continuous function. 
\end{defn}

Assembly maps are incredibly useful to the point of view offered here, so we review their construction. 

\begin{lemma} \label{contthenass}
	If $F$ is a continuous functor, then $F$ has \emph{assembly}, a binatural transformation which is also natural in $F$ given by
	\[\alpha_F:Z \sm F(X) \longrightarrow F(Z \sm X).\]
\end{lemma}

\begin{proof}
	The assembly map is given by pushing the identity through the following maps
\begin{align*}
\Hom(Z \sm X, Z \sm X) \xrightarrow{\  \cong\ } &\Hom (Z, \Hom(X, Z \sm X)) \\
\xrightarrow{F_\ast} &\Hom(Z, \Hom(F(X), F(Z \sm X))) \\
\xrightarrow{\ \cong\ } &\Hom(Z \sm F(X), F(Z \sm X))
\end{align*}\end{proof}

Note that this proof requires $\Hom(X,Y) \to \Hom(F(X),F(Y))$ to be a pointed map, and so $X \to \ast \to Y$ must be sent to the basepoint of $\Hom(F(X),F(Y))$, thus a functor $F$ must be pointed in order to be continuous.


\begin{defn}
	A functor $F:\cC \to \cC$ is a \emph{monad} if it is a monoid in the category of endofunctors on $\cC$. More explicitly, a monad $F$ is a functor equipped with natural transformations $\eta: \Id \to F$ and $\gamma: F \circ F \to F$ satisfying associativity and unitality diagrams.
\end{defn}

\subsection{Goodwillie's linearization}

In \cite{G3}, Goodwillie constructs the Taylor tower $\{P_nF\}$ of a homotopy functor $F$ from topological spaces to spaces or spectra, and Kuhn shows that Goodwillie's work extends to functors between more general model categories \cite{kuhnoverview}. We will concentrate on the linearization of an endofunctor of spaces.

\begin{defn}
	A functor $F$ is \emph{1-excisive} if for every homotopy cocartesian square $\cX$, the square $F(\cX)$ is homotopy cartesian.
\end{defn}

We will often omit the word ``homotopy'' from our limits and colimits, but it is always intended, unless noted otherwise.

Goodwillie defines the $1$-excisive approximation $P_1F$ of a homotopy functor $F$ as the homotopy colimit of an infinite iteration of intermediate functors $T_1F$. For the purposes of this paper, we only need $T_1F$ for reduced functors $F$, so we substitute Goodwillie's definition with the following equivalent one.

\begin{defn} \label{T1defn}
If $F$ is reduced, the functor $T_1F$ is equivalent to $\Omega \circ F \circ \Sigma$.
\end{defn}

We see that $T_1: \Fun(\cT, \cT) \to \Fun(\cT, \cT)$ can be iterated, and thus $T_1^iF \simeq \Omega^i \circ F \circ \Sigma^i$ for $F$ reduced. There is a natural transformation $t_1:F \to T_1F$. Then $P_1F = \hc_{i \to \infty} T_1^iF$, and the 1st \emph{layer} of the Taylor tower is the fiber $D_1F= \hofib [ P_1F \to F(\ast)]$. The functor $D_1F$ is reduced and is called the \emph{linearization} of the functor $F$. When $F$ is reduced, $P_1F$ is equivalent to the linearization, and $D_1F \simeq \hc_i \Omega^i F \Sigma^i$.

\begin{defn}
	Let $F:\cT \to \cT$ be a homotopy functor.
$F$ is \emph{stably $1$-excisive} or satisfies stable 1st order excision, if the following condition holds for some numbers $c$ and $\kappa$:

$E_1(c,\kappa)$: If $\cX: \cP(\{0,1\}) \to \cC$ is any strongly cocartesian square such that the maps $\cX(\nulll) \to \cX(\{s\})$ are $k_s$-connected for $s \in \{0,1\}$ and $k_s \geq \kappa$, then the diagram $F(\cX)$ is $(-c+\Si k_s)$-cartesian.
\end{defn}

\begin{ex}
	\cite[4.3, 4.5]{G2} The identity functor of spaces is stably 1-excisive by the Blakers-Massey theorem, and the functor $\Hom(K,-)$ is stably 1-excisive, satisfying $E_1(kn, -1)$ where $k=dim(K)$.
\end{ex}



We think of stable 1-excision as being a bound on the failure of a functor to be 1-excisive; that is, applying the functor to a cocartesian square of sufficiently connected spaces is predictably cartesian. These connectivity conditions turn out to be precisely what is necessary to linearize over a different category and keep the same homotopy, as indicated in Lemma \ref{pnfagree}. 

\subsection{Symmetric sequences and symmetric functor sequences}
We review the pertinent definitions of symmetric sequences and the related category of functors. 

\begin{defn}
	 Let $\cC$ be a category and let $\Sigma$ be the category of finite sets and bijections. A \emph{symmetric sequence} in $\cC$ is a functor $A: \Si \to \cC$. This is a sequence $\{ A(n) \}_{n \geq 1}$ of objects of $\cC$ with a $\Si_n$-action on $A(n)$ for each $n \geq 1$. A morphism of symmetric sequences $f:A \to B$ is a natural tranformation of functors or, explicitly, a sequence of $\Si_n$-equivariant morphisms $f(n):A(n) \to B(n)$. We denote the category of symmetric sequences in $\cC$ by $\Symseq(\cC)$.
\end{defn}

\begin{defn}
	If $\cC$ is a cocomplete closed symmetric monoidal category with monoidal product denoted $\sm$ and if $A,B$ are symmetric sequences in $\cC$, then the \emph{composition product} or \emph{$\circ$-product} of $A$ and $B$ is the symmetric sequence $A \circ B$ defined by 
	\[
	(A \circ B) (n) = \bigvee_{\text{unordered partitions of } \{1, \dots, n \}} A(k) \sm B(n_1) \sm \cdots \sm B(n_k).
	\] 
\end{defn}



The composition product defines a monoidal product on the category of symmetric sequences in $\cC$. If the unit of $\cC$ is $S$ and zero-object $\ast$, the unit object of $\Symseq(\cC)$ is given by 
\[
1(n)= \left\{ \begin{array}{lc} 
S & \text{if } n=1 \\ \ast & \text{else} 
\end{array} \right.
\]

\begin{defn}
	An \emph{operad} in $\cC$ is a monoid in $\Symseq(\cC)$ under the composition product; that is, an operad is a symmetric sequence $\cO$ with a composition map $\gamma:\cO \circ \cO \to \cO$ and a unit map $\eta:1 \to \cO$ satisfying associativity and unitality diagrams.
\end{defn}

\begin{defn}
Let $\cC$ be a symmetric monoidal category and let $\cC^k$ denote the $k$-fold cartesian product. Each permutation $\sigma \in \Sigma_k$ yields a map $\sigma_\#:\cC^k \to \cC^k$ which permutes the coordinates, $\sigma_\#(X_1, \dots, X_k)=(X_{\sigma(1)}, \dots, X_{\sigma(k)})$.
	A \emph{symmetric functor sequence} in $\cC$ is a sequence of functors $\cF_k: \cC^k \to \cC$ with natural isomorphisms $\sigma_\ast: \cF_k \to \cF_k \circ \sigma_\#$ for each $\sigma \in \Sigma_k$. A morphism of symmetric functor sequences is a sequence of levelwise natural transformations. We denote the category of symmetric functor sequences in $\cC$ by $\Symfun(\cC)$.
\end{defn}

This category is related to a 2-category introduced in \cite{daystreet}. It is shown there that if $\cC$ has coproducts, $\Symfun(\cC)$ has a monoidal product defined for symmetric functor sequences $\cF_\ast$ and $\cG_\ast$ by 
\[
(\cG \circ \cF)_n = \sum_{j_1 + \cdots +j_k=n} \cG_k \circ (\cF_{j_1} \times \cdots \times \cF_{j_k}).
\]

The unit of this monoidal product is given by the initial object $0 \in \Fun(\cC^n,\cC)$ for $n \neq 1$ and $\Id_\cC$ for $n=1$.

\begin{defn}
	A monoid in $\Symfun(\cC)$ is called a \emph{multitensor}, a \emph{lax monoidal category} (in \cite{daystreet}), or a \emph{functor-operad} (in \cite{McClureSmith}).
\end{defn}

\begin{ex}

\begin{enumerate}[(a)]
	\item In a symmetric monoidal category, $\bigotimes_k(X_1, \dots X_k)= \bigotimes_{i=1}^k X_k$ defines a functor-operad. If the category is pointed, $\bigotimes_k$ is pointed in each variable.
	\item If $\op$ is an operad in a category $\cC$ with products, then $\prod_k^\op(X_1, \dots, X_k)=\op(k) \times X_1 \times \cdots \times X_k$ is a functor-operad.
\end{enumerate}
\end{ex}

In \cite{bjy}, Bauer, Johnson, and the author show that for a monad $F$ of $R$-modules, the directional derivatives of abelian calculus, $\nabla^k F_R(X):\Mod_R^k \to \Mod_R$, form a functor-operad. 


\begin{defn} \label{defn:monfun}
	A functor $F:\cC \to \sD$ between monoidal categories $(\cC, \otimes_\cC, 1_\cC)$ and $(\sD, \otimes_\sD, 1_\sD)$ is \emph{monoidal} if there is a morphism $\epsilon: 1_\sD \to F(1_\cC)$ and a natural tranformation $\mu_{X,Y}:F(X) \otimes_\sD F(Y) \to F(X \otimes_\cC Y)$ satifying associativity and unitality diagrams.
\end{defn}

\section{A new linearization}

In this section, we define a new model for linearization, using the sphere operad and the category of finite sets with injections. 

\subsection{The category $\II$}
In redefining linearization, we will exploit the properties of a particular indexing category used by B\"{o}kstedt in his definition of topological Hochschild homology. The use of the category $\II$ has found great success in the areas of algebraic K-theory and representation stability. 

\begin{defn}
	Let $\II$ denote (the skeleton of) the category of finite sets and injective maps. Let $\NN$ denote the category of finite sets with only the standard inclusions (those induced by subset inclusion). 
\end{defn}

B\"{o}kstedt showed that under certain conditions on a functor $G: \II \to \cT$, $\hc_\NN G \to \hc_\II G$ is an equivalence \cite{bokthh}. Essentially, the condition is that maps further in the diagram become more and more connected. 
For a multifunctor, there is a criterion for equivalence which reduces to that of B\"{o}kstedt's when $q=1$. 

\begin{lemma} \cite[2.2.2.2]{Dundas-Good-McCarthy} \label{dgm} \label{bok}
	If $G:\II^{q} \to \cT$, $\mathbf{x} \in \II^q$, let $\mathbf{x} \downarrow \II^q$ be the full subcategory of $\II^q$ receiving maps from $\mathbf{x}$, then if $G$ sends maps in $\mathbf{x} \downarrow \II^q$ to $n_{|\mathbf{x}|}$-connected maps and $n_{|\mathbf{x}|} \rightarrow \infty$ as $|\mathbf{x}| \rightarrow \infty$, then $\hc_{\NN^q} G \to \hc_{\II^q} G$ is an equivalence. 
\end{lemma}

Using a homotopy colimit over $\II$, we can define a tower $\{\PP_nF\}$ which is equivalent to the Taylor tower when $F$ is analytic. We prove the $n=1$ case. For functors $F:\cT \to \Sp$, the Taylor tower defined in this way is investigated further in \cite{sayclass}.

\begin{defn} \label{PnF} Let $\PP_1F=\hc_{k \in \II} T_1^k F$.
\end{defn}

\begin{lemma} \label{pnfagree}
	When $F$ is stably $1$-excisive, $P_1F \to \PP_1F$ is an equivalence.
\end{lemma}

\begin{proof}
We will show that the functor $\Theta: \II \to Fun(\cT, \cT)$ defined by $\Theta(\mathbf{k})=T_1^{|\mathbf{k}|} F$ satisfies the hypotheses of B\"okstedt's lemma (\ref{bok}) when $F$ satisfies $E_1(c, \kappa)$. By \cite[Prop 1.4]{G3}, if $F$ satisfies $E_1(c, \kappa)$, then $T_1F$ satisfies $E_1(c-1, \kappa -1)$ and $t_1F: F \to T_1F$ is $(-c+2\ell)$-connected on objects $X$ which are ($\ell-1$)-connected with $\ell\geq k$.
By induction on $i$, $T_1^{i} F$ satisfies $E_1(c-i, \kappa-i)$, and $T_1^{i} F(X) \to T_1^{i+1}F(X)$ is at least $(i-c+2 \ell)$-connected for $\ell \geq \kappa$. Since $(i-c+2\ell)$ increases as $i$ increases, $\Theta$ satisfies the condition of B\"{o}kstedt's lemma. 
\end{proof}

\subsection{The sphere operad}
We will need to be careful with the model of spheres we use in our linearizations, as we need strict associativity. 
We will make use of the sphere operad defined in \cite{AKsphere}, so we recall its definition and salient properties now.

The sphere operad $\mathbf{S}$ is the one-point compactification of a nonunital simplex operad, whose $n$th space is the open $(n-1)$-dimensional simplex, so the $n$th space of $\mathbf{S}$ is homeomorphic to $S^{n-1}$. The operad composition maps are homeomorphisms
\[
S^{k-1} \sm S^{j_1-1} \sm \cdots \sm S^{j_k-1} \to S^{j_1 + \cdots + j_k -1}.
\]

There is a map of operads $\mathbf{S} \to \text{Coend}(S^1)$ such that for each $n \geq 1$ the map $\mathbf{S}_n\cong S^{n-1} \to \Omega S^n$ is adjoint to a homeomorphism $S^{n-1} \sm S^1 \to S^n$. Since the $\Si_n$-action on the coendomorphism operad of $S^1$ permutes the $n$ coordinates of $S^n$, this defines a $\Si_n$-equivariant map $S^1 \sm \mathbf{S}_n \cong S^n$. Finally, there is a map of operads $Com \to \mathbf{S}$ such that the composite $Com \to \mathbf{S} \to \text{Coend}(S^1)$ is levelwise the canonical map adjoint to the diagonal map $S^1 \to S^n$.

\begin{defn}
	Let $\mathbf{S}^U$ denote the operad whose $n$th space is the smash product of $U$ copies of $\mathbf{S}_n$.
\end{defn}

We use the finite set $U$ to index the spheres because we will be linearizing over the category $\II$. 
The operad $\mathbf{S}^U$ has the diagonal $\Si_n$-action induced by that on $\mathbf{S}_n$, and composition maps require a shuffling of coordinates before applying the composition maps of $\mathbf{S}$.

\begin{prop} \label{magic}
For a continuous functor $F:\cT \to \cT$ and $j \geq 1$, there is an associative map
\[
\Omega^U F(S^U) \to \Omega^{\coprod_j U} F(S^{\coprod_j U}). 
\]
\end{prop}

Define the map by smashing an element of $\Omega^U F(S^U)$ with the $j$th space of $\mathbf{S}^U$ then assembling the sphere into $F$. That is, $f$ in $\Omega^U F(S^U)$ maps to the composite
\[
S^{\coprod_j U} \cong \mathbf{S}^U_j \sm S^U \xrightarrow{\mathbf{S}^U_j \sm f} \mathbf{S}^U_j \sm F(S^U) \xrightarrow{\alpha_F} F(\mathbf{S}^U_j \sm S^U) \cong F(S^{\coprod_j U}).
\]

This map is strictly associative because the sphere operad composition maps are associative. For example, under the maps $\Omega S^1 \to \Omega^2 S^2 \to \Omega^3 S^3$, the element $f \in \Omega S^1$ is sent to $\mathbf{S}_2 \sm \mathbf{S}_2 \sm \mathbf{S}_1 \sm f$, which is equivariantly homeomorphic to $\mathbf{S}_3 \sm f$, the image of $f$ under $\Omega S^1 \to \Omega^3 S^3$.

\section{Linearizing symmetric functor sequences} \label{mainthm}

In \cite{G3}, Goodwillie identifies the derivatives of a functor as the multilinearized cross-effects evaluated at $S^0$. The $n$th cross-effect is a functor of $n$ variables which satisfies some connectivity hypotheses, namely, it is stably 1-excisive in each variable, even after partial linearization. The collection of cross-effects form a symmetric functor sequence. In this section, we will consider multilinearization of more general symmetric functor sequences.

\begin{defn}
	Let $\cF_n:\cT^n \to \cT$. Then we denote the linearization of $\cF_n$ in each variable over $\II$ by 
	\[\mathbb{D}_1^{(n)}\cF_n(X_1, \dots, X_n) =\hc_{U_1, \dots, U_n \in \II} \Omega^{U_1} \cdots \Omega^{U_n} \cF_n(\Sigma^{U_1}X_1, \dots, \Sigma^{U_n}X_n).\] 
	We denote the $\II^n$-shaped diagram over which the linearization is taken by $\Omega^{\bullet, \cdots, \bullet} \cF_n \Sigma^{\bullet, \cdots, \bullet}$.
\end{defn}

There is a natural map $D_1^{(n)}\cF_n \to \mathbb{D}_1^{(n)}\cF_n$, which is an equivalence when $\Omega^{\bullet, \cdots, \bullet}\cF_n \Sigma^{\bullet, \cdots, \bullet}$ satisfies the condition of Lemma \ref{dgm}. We abuse notation by letting $\cF_n(X)=\cF_n(X, \dots, X)$, the evaluation of the multifunctor on the diagonal. 
%

\begin{thm} \label{linisoperad}
	The functor $\mathbb{D}_1^{(\ast)}-(S^0):\Symfun_\ast(\cT) \to \Symseq(\cT)$ from multipointed symmetric functor sequences of spaces to symmetric sequences is lax monoidal.
\end{thm}

\begin{proof}
The unit $\epsilon: S^0 \to \DD_1 \Id(S^0)$ is given by inclusion of the first object in the homotopy colimit
\[
\Id(S^0) \to \DD_1 \Id (S^0)=\hc_\II \left( \xymatrix{\Id(S^0) \ar[r] & \Omega \Id \Sigma (S^0) \ar@<-.5ex>[r] \ar@<.5ex>[r] & \cdots}  \right)
\]

We note that the symmetric group $\Sigma_n$ acts on $\DD_1^{(n)} \cF_n(S^0)$ by permuting the $n$ inputs of $\cF_n$. In the linearization, this also block-permutes the loops.


We now define 
\[
\mu_{k,j_i}:\DD_1^{(k)}\cG_k(S^0) \sm \DD_1^{(j_1)}\cF_{j_1}(S^0) \sm \cdots \sm \DD_1^{(j_k)} \cF_{j_k}(S^0) \to \DD_1^{(\sum j_i)} (\cG \circ \cF)_{\sum j_i}(S^0)
\]

We will start with the definition of the map on level one:
\[
\mu_{1,1}:\DD_1\cG_1(S^0) \sm \DD_1\cF_{1}(S^0) \to \DD_1 (\cG \circ \cF)_1 (S^0).\] Recall that the homotopy colimit and loops functors are both continuous, so by Lemma \ref{contthenass} they have assembly maps $\alpha$. The first step is to assemble the homotopy colimits and loops out of the smash product. Next, we use assembly for $\cF_1$ and $\cG_1$ to nest them then include this summand into the composition $\cG \circ \cF$. Finally, we use the map induced by $\coprod: \II \times \II \to \II$ to reindex the homotopy colimit.

\[
\xymatrix{\hc \limits_{U \in \II} \Omega^U \cG_1(S^U) \sm \hc \limits_{V \in \II} \Omega^V \cF_1(S^V) \ar[d]^{\alpha_{\hc}, \,  \alpha_\Omega} \\ 
\hc \limits_{U \in \II} \hc \limits_{V \in \II} \Omega^U  \Omega^V \cG_1(S^U) \sm  \cF_1(S^V) \ar[d]^{\alpha_{\cG_1},\alpha_{\cF_1}} \\
\hc \limits_{U \in \II} \hc \limits_{V \in \II} \Omega^U  \Omega^V \cG_1(\cF_1(S^U \sm S^V)) \ar[d]^{\text{incl}} \\
\hc \limits_{(U,V) \in \II \times \II} \Omega^{U\coprod V} (\cG \circ \cF)_1(S^{U\coprod V}) \ar[d]^{\coprod_\ast} \\
\hc \limits_{W \in \II} \Omega^W  (\cG \circ \cF)_1(S^W)}
\]

\begin{rem} The last step is the key reason for using $\II$; if the homotopy colimit is defined over $\NN$, the map $\mu$ can be defined, but it will not be strictly associative on homotopy colimits. This is similar to the reason naive spectra do not have a good smash product, but symmetric spectra have enough extra structure to encode the smash product in an associative way.
	 \end{rem}

We will introduce new notation to save some ink in the definition of the general $\mu$. If $U, V_1, \dots , V_{k}$ are finite sets, let $S^{\underline{V}}$ denote the $k$-tuple of spheres $(S^{V_1}, \dots, S^{V_k})$ and let $S^{U \coprod \underline{V}}=(S^{U\coprod V_1}, \dots , S^{U \coprod V_k})$.
We will restrict to the case 
\[\mu_{2,j_1,j_2}:\DD_1^{(2)}\cG_2(S^0) \sm \DD_1^{(j_1)} \cF_{j_1}(S^0) \sm \DD_1^{(j_2)} \cF_{j_2}(S^0) \to \DD_1^{(j_1+j_2)} (\cG \circ \cF)_{j_1+j_2}(S^0),\]
and note that the general case follows easily. 

The map $\mu$ is defined as a long composition, with most maps the same as in the level 1 case. As before, we assemble the homotopy colimits and loops out of the smash product first, but then we apply the map constructed in Proposition \ref{magic}, $\al_{\cG_2} \circ \left[ \mathbf{S}^{U_1}_{j_1} \sm \mathbf{S}^{U_2}_{j_2} \sm -\right]$. Then we use assembly for the $\cF_{j_i}$ to nest them, noting that only one copy of $S^{U_i}$ assembles into each variable of $\cF_{j_i}$. Finally, as before, we include into the composition of symmetric functor sequences, and use the coproduct of $\II$ to reindex the homotopy colimit. That is, $\mu$ is defined by the following composite.

\[ \xymatrix{
\makebox[.8\textwidth][c]{$\hc \limits_{U_1, U_2 \in \II} \Omega^{U_1}\Omega^{U_2} \cG_2 (S^{U_1}, S^{U_2}) \sm \hc \limits_{V_1^1, \dots, V_{j_1}^1 \in \II} \Omega^{\coprod V^1_i} \cF_{j_1} (S^{\underline{V}^1}) \sm \hc \limits_{V_1^{2}, \dots V_{j_2}^2 \in \II} \Omega^{\coprod V_\ell^2} \cF_{j_2} (S^{\underline{V}^2})$} \ar[d]^{\alpha_{\hc}, \alpha_\Omega} \\ \hc \limits_{U_1, U_2, V_1^1, \dots, V_{j_2}^2 \in \II}  \Omega^{U_1}\Omega^{U_2} \Omega^{\coprod V^1_i} \Omega^{\coprod V_\ell^2} \cG_2 (S^{U_1},S^{U_2}) \sm   \cF_{j_1} (S^{\underline{V}^1}) \sm   \cF_{j_2}(S^{\underline{V}^2}) \ar[d]^{\al_{\cG_2} \circ \left[ \mathbf{S}^{U_1}_{j_1} \sm \mathbf{S}^{U_2}_{j_2} \sm -\right]} \\ 
\makebox[.8\textwidth][c]{$\hc \limits_{U_1, U_2, V_1^1, \dots, V_{j_2}^2 \in \II}  \Omega^{\coprod_{j_1} U_1}\Omega^{\coprod_{j_2} U_2} \Omega^{\coprod V^1_i} \Omega^{\coprod V_\ell^2} \cG_2 (S^{\coprod_{j_1}U_1}, S^{\coprod_{j_2}U_2}) \sm   \cF_{j_1} (S^{\underline{V}^1}) \sm   \cF_{j_2} (S^{\underline{V}^2})$} \ar[d]^{\alpha_{\cG_2}} \\
\makebox[.8\textwidth][c]{$\hc \limits_{U_1, U_2, V_1^1, \dots, V_{j_2}^2 \in \II} \Omega^{\coprod_{j_1} U_1}\Omega^{\coprod_{j_2} U_2} \Omega^{\coprod V^1_i} \Omega^{\coprod V_\ell^2} \cG_2 (S^{\coprod_{j_1}U_1}\sm   \cF_{j_1} (S^{\underline{V}^1}), S^{\coprod_{j_2}U_2} \sm  \cF_{j_2} (S^{\underline{V}^2}))$} \ar[d]^{\alpha_{\cF_{j_i}}} \\
\makebox[.8\textwidth][c]{$\hc_{U_1,U_2, V_1^1, \dots, V_{j_2}^2 \in \II}  \Omega^{\coprod (U_1 \coprod V^1_i)} \Omega^{\coprod (U_2 \coprod V_\ell^2)} \cG_2 ( \cF_{j_1} (S^{U_1 \coprod \underline{V}^1}), \cF_{j_2} (S^{U_2 \coprod \underline{V}^2}))$} \ar[d]^{\coprod_\ast \circ \text{incl}} \\
\hc_{W_1, \dots, W_{j} \in \II}  \Omega^{\coprod W_i} (\cG \circ \cF)_j( S^{W_1}, \dots, S^{W_{j}})
} 
\]

Note that the assembly maps are equivariant and associative, as is the map described in Proposition \ref{magic}, and the composition of symmetric functor sequences is also equivariant with respect to permuting the variables, so the composition above is equivariant and associative.
\end{proof}

Note further that if a functor-operad $\cF_\ast$ is such that $\Omega^{\bullet, \cdots, \bullet} \cF_n \Sigma^{\bullet, \cdots, \bullet}$ satisfies the conditions of Lemma \ref{dgm} for each $n$, then the multilinearization over $\II$ is equivalent to the usual multilinearization of $\cF_\ast$.

Monoidal functors take monoids to monoids, so the multilinearization of a multipointed functor-operad or multitensor of spaces evaluated at $S^0$ is an operad. Given a monoidal model for the cross-effects functor $cr_\ast$, this would recover the operad structure of the derivatives of the identity functor of spaces, giving explicit structure maps and extend the operad structure for the identity in \cite{Chingthesis} to all monads on $\cT$.

\bibliography{mybib}{}
\bibliographystyle{amsalpha}

\end{document}